\documentclass[reqno,11pt]{amsart}
\usepackage{amsthm, amsmath, amssymb}
\usepackage[utf8]{inputenc}
\usepackage[T1]{fontenc}

\usepackage[hyphens]{url}
\usepackage{systeme}
\usepackage[shortlabels]{enumitem}
\usepackage[hidelinks]{hyperref}
\usepackage{microtype}
\usepackage[numbers,sort]{natbib} 
\usepackage{bm}
\usepackage[margin=1in]{geometry}
\usepackage{tikz}
\usetikzlibrary{calc}

\usepackage[textsize=scriptsize,backgroundcolor=orange!5]{todonotes}

\usepackage[noabbrev,capitalize,sort]{cleveref}
\crefname{equation}{}{}
\crefname{enumi}{}{}
\numberwithin{equation}{section}

\usepackage{mathtools}
\newtheorem{theorem}{Theorem}[section]
\newtheorem{proposition}[theorem]{Proposition}
\newtheorem{lemma}[theorem]{Lemma}

\theoremstyle{definition}
\newtheorem{definition}[theorem]{Definition}

\theoremstyle{remark}
\newtheorem*{remark}{Remark}

\newcommand{\abs}[1]{\left\lvert#1\right\rvert}

\newcommand\up[1]{^{(#1)}}

\DeclareMathOperator{\Left}{left}  
\DeclareMathOperator{\Right}{right}
\DeclareMathOperator{\twr}{twr}
\DeclareMathOperator{\PG}{PG}

\newcommand*{\eqdef}{\stackrel{\mbox{\normalfont\tiny def}}{=}}

\newcommand{\FF}{\mathbb{F}}

\newcommand*{\PP}{\mathbb{P}}

\newcommand{\cE}{\mathcal E}

\newcommand{\cT}{\mathcal{T}}

\newlength{\hght}

\renewcommand\Tilde[1]{\widetilde{#1}}

\makeatletter
\newcommand\thankssymb[1]{\textsuperscript{\@fnsymbol{#1}}}
\makeatother

\author[Xiaoyu He]{Xiaoyu He\thankssymb{1}}

\thanks{\thankssymb{1} School of Mathematics, Georgia Institute of Technology, Atlanta, GA 30332\@. Email: {\tt xhe399@gatech.edu}}

\author[Jiaxi Nie]{Jiaxi Nie\thankssymb{2}}
\thanks{\thankssymb{2} School of Mathematics, Georgia Institute of Technology, Atlanta, GA 30332\@. Email: {\tt jnie47@gatech.edu}}

\author[Yuval Wigderson]{Yuval Wigderson\thankssymb{3}}

\thanks{\thankssymb{3} Institute for Theoretical Studies, ETH Z\"urich, 8006 Z\"urich, Switzerland. Email: {\tt yuval.wigderson@eth-its. ethz.ch}. Research supported by Dr.\ Max R\"ossler, the Walter Haefner Foundation, and the ETH Z\"urich Foundation.}

\author[Hung-Hsun Hans Yu]{Hung-Hsun Hans Yu\thankssymb{4}}

\thanks{\thankssymb{4}Department of Mathematics, Princeton University, Princeton, NJ 08544\@. Research supported by NSF grant DMS - 2246682. Email: {\tt hansonyu@princeton.edu}}

\title{Off-diagonal Ramsey numbers for linear hypergraphs}

\begin{document}

\begin{abstract}
We study off-diagonal Ramsey numbers $r(H, K_n^{(k)})$ of $k$-uniform hypergraphs, where $H$ is a fixed linear $k$-uniform hypergraph and $K_n^{(k)}$ is complete on $n$ vertices. Recently, Conlon et al.\ disproved the folklore conjecture that $r(H, K_n^{(3)})$ always grows polynomially in $n$. In this paper we show that much larger growth rates are possible in higher uniformity. In uniformity $k\ge 4$, we prove that for any constant $C>0$, there exists a linear $k$-uniform hypergraph $H$ for which 
\[
    r(H,K_n\up k) \geq \twr_{k-2}(2^{(\log n)^C}).
\]
 
\end{abstract}
\maketitle

\section{Introduction}
Let $H_1,H_2$ be two $k$-uniform hypergraphs ($k$-graphs for short). Their \emph{Ramsey number} $r(H_1,H_2)$ is defined as the least integer $N$ such that every red/blue edge-coloring of the complete $k$-graph $K_N\up k$ contains a monochromatic red copy of $H_1$ or a monochromatic blue copy of $H_2$. In this paper, we will be concerned with the \emph{off-diagonal} regime, where $H_1=H$ is a fixed $k$-graph, and $H_2=K_n\up k$ is a clique, and we are interested in the asymptotics of $r(H, K_n\up k)$ as $n \to \infty$. Note that $r(H, K_n\up k)$ can also be viewed as the least integer $N$ such that every $N$-vertex $H$-free $k$-graph contains an independent set of size $n$.

In the special case of $k=2$, it follows from the classical Erd\H os--Szekeres bound \cite{MR1556929} that $r(H,K_n\up 2) = n^{\Theta_H(1)}$, i.e.\ that $r(H,K_n\up 2)$ grows polynomially in $n$ for any fixed $H$. However, our knowledge of the correct order of polynomial growth is limited to extremely few graphs $H$. For example, while an early result of Erd\H os \cite{MR120168} implies that $r(K_3\up 2, K_n\up 2) = n^{2+o(1)}$ (with many further developments \cite{MR0600598,MR708165,MR1369063,MR4073152,MR4201797,2505.13371} obtaining much more precise asymptotic bounds), it was only extremely recently that Mattheus and Verstra\"ete \cite{MR4713025} proved that $r(K_4\up 2, K_n\up 2) = n^{3+o(1)}$. It remains a major open problem to determine the order of polynomial growth of $r(H,K_n\up 2)$ for most other graphs $H$, for example when $H$ is a clique of order at least $5$, or a cycle of length at least $4$.

For hypergraphs of uniformity $k \geq 3$, our understanding is even more limited. For a general fixed $H$, the best known upper bound on $r(H,K_n\up k)$ follows from the stepping-down technique of Erd\H os and Rado \cite{MR0065615}, which implies that
\begin{equation}\label{eq:twr UB}
    r(H,K_n\up k) \leq \twr_{k-1}(n^{O_H(1)}),
\end{equation}
where the \emph{tower function} is recursively defined by $\twr_1(x)=x$ and $\twr_i(x) = 2^{\twr_{i-1}(x)}$ for $i \geq 2$.

Perhaps surprisingly, the upper bound \eqref{eq:twr UB} turns out to be tight (up to the implicit constant) in many cases, including when $H$ is a clique of order at least $k+2$ \cite{MR3830113}. The proof of this result, and of many other lower bounds \cite{MR2552253,MR3030610,MR3641804,2404.02021} on hypergraph Ramsey numbers, uses a variant of the celebrated stepping-up construction of Erd\H os, Hajnal, and Rado \cite{MR0202613}, which we shortly discuss in more detail.

Although the bound in \eqref{eq:twr UB} is tight when $H$ is a clique, it is natural to expect that it is far from tight when $H$ is ``far from being a clique''. For example, a simple supersaturation argument, going back at least to the early work of Erd\H os and Hajnal \cite{MR337636}, demonstrates that $r(H,K_n\up k)=n^{O_H(1)}$ whenever $H$ is $k$-partite, or, more generally, \emph{iterated} $k$-partite\footnote{The class of iterated $k$-partite hypergraphs is defined recursively as follows. First, a hypergraph with no edges is iterated $k$-partite. Next, $H$ is iterated $k$-partite if $V(H)$ can be partitioned into $k$ parts $V_1,\dots,V_k$ such that every edge either transverses the $k$ parts or is fully contained in one part, and $H[V_i]$ is iterated $k$-partite for all $i$.}. It is conjectured in \cite{2411.13812} that this is in fact a complete characterization: $H$ is iterated $k$-partite if and only if $r(H,K_n\up k) = n^{O_H(1)}$. 

Towards resolving this conjecture, or more generally towards understanding the behavior of $r(H,K_n\up k)$ for general $H$, it is natural to ask what happens when $H$ is still ``far from clique-like'', but is not necessarily iterated $k$-partite. One such class of hypergraphs is the class of \emph{linear} hypergraphs, in which every pair of edges intersects in at most one vertex. A longstanding folklore conjecture in the field was that in uniformity $k=3$, every linear $3$-graph $H$ satisfies $r(H,K_n\up 3) = n^{O_H(1)}$. However, this conjecture was recently disproved in \cite{2404.02021}.
\begin{theorem}[{\cite[Theorem 1.4]{2404.02021}}]\label{thm:linear LB 3-uniform}
    For every $C>1$, there exists a linear $3$-graph $H$ such that $r(H,K_n\up 3) \geq 2^{(\log n)^C}$ for all sufficiently large $n$.
\end{theorem}
In particular, this lends some credence to the conjecture of \cite{2411.13812}: since linear hypergraphs need not have polynomial off-diagonal Ramsey numbers, perhaps the only relevant structure is that of being iterated $k$-partite. 

Our main result is a generalization of \cref{thm:linear LB 3-uniform} to arbitrary uniformities. 
\begin{theorem}\label{thm:main}
    For every $C>1$ and every $k \geq 3$, there exists a linear $k$-graph $H$ such that
    \[
        r(H,K_n\up k) \geq \twr_{k-2}(2^{(\log n)^C})
    \]
    for all sufficiently large $n$. 
\end{theorem}
We note that the tower height in \cref{thm:main} is nearly best possible, by \eqref{eq:twr UB}. We also remark that, if we let $K_{n,\dots,n}\up k$ denote the complete $k$-partite $k$-graph with parts of size $n$, then a simple supersaturation argument (see \cite[Proposition 7.2]{MR4332486}) demonstrates that $r(H,K_{n,\dots,n}\up k) \leq n^{O_H(1)}$ for any linear $k$-graph $H$. Thus, \cref{thm:main} implies that $r(H,K_n\up k)$ and $r(H,K_{n,\dots,n}\up k)$ can be extremely far apart, in that the latter is polynomial and the former grows as a tower of height at least $k-2$. Previously no such separation was known for any $H$; in particular, if $H$ is a fixed clique, then $r(H,K_n\up k)$ and $r(H,K_{n,\dots,n}\up k)$ are expected to be of roughly the same order.

In order to prove \cref{thm:main}, we need to both define the linear $k$-graph $H$ and to define a coloring of $E(K_N\up k)$ containing no red copy of $H$ and no large blue clique. As in many previous works, our coloring is based on the stepping up construction of Erd\H os, Hajnal, and Rado \cite{MR0202613}. In particular, our technique takes as input a linear $(k-1)$-graph $F$, and outputs a linear $k$-graph $H$ whose off-diagonal Ramsey number is exponential in that of $F$, as stated in the following proposition.
\begin{proposition}\label{prop:step up linear}
    Let $k \geq 4$ and let $H$ be a linear $(k-1)$-graph. There exists a linear $k$-graph $H'$ such that
    \[
        r(H',K_{2n+2k}\up k) > 2^{r(H,K_n\up{k-1})-1}.
    \]
\end{proposition}
Note that \cref{prop:step up linear} immediately implies \cref{thm:main} by induction on $k$, with \cref{thm:linear LB 3-uniform} serving as the base case.

As discussed above, the coloring used in the proof of \cref{prop:step up linear} is based on the stepping up technique. Although variants of this technique appear frequently in the literature, most papers simply record certain ad hoc properties needed for the particular application. We therefore present in \cref{sec:stepping up} a slightly more general treatment of the properties of this construction, in the hopes that future researchers can use some of these in a black-box way. We then present the construction of the linear $k$-graph $H$ in \cref{sec:linear}, and complete the proof of \cref{prop:step up linear}, and thus of \cref{thm:main}.

\section{Stepping up}\label{sec:stepping up}
In classical stepping-up constructions, given a $(k-1)$-graph $G$ on $\{0,\ldots,N-1\}$, one constructs a $k$-graph $\Tilde{G}$ whose vertices are all vectors $\varepsilon=(\varepsilon_{N-1},\varepsilon_{N-2},\ldots, \varepsilon_0)\in \{0,1\}^N$, ordered lexicographically.
One also considers, for any two distinct strings $\varepsilon,\varepsilon'$, the largest index $\delta=\delta(\varepsilon,\varepsilon')$ where they differ.
Observe that $\varepsilon<\varepsilon'$ if and only if $\varepsilon_{\delta}<\varepsilon'_{\delta}$ where $\delta = \delta(\varepsilon, \varepsilon')$.
Finally, for any $k$ strings $\varepsilon^{(1)}<\cdots <\varepsilon^{(k)}$, we decide whether they form an edge in $\Tilde{G}$ by some deterministic rule about the tuple $(\delta_1,\ldots,\delta_{k-1})$ where $\delta_i =\delta(\varepsilon^{(i)},\varepsilon^{(i+1)})$.
The exact rule depends on the application, though typically if $\delta_1,\ldots, \delta_{k-1}$ is monotone, then $\{\varepsilon^{(1)},\ldots, \varepsilon^{(k)}\}$ is an edge if and only if $\{\delta_1,\ldots,\delta_{k-1}\}$ is an edge in $G$.

In this paper, we switch to a new set of notations that represents each vertex in the stepping-up construction with a single number instead of a string.
We will also rephrase the conditions on $\delta_1,\ldots, \delta_{k-1}$ in terms of binary trees, which should allow one to visualize the constructions better.

First of all, suppose that $a,a'$ are two distinct numbers in $\{0,1,\ldots, 2^N-1\}$. If their binary representations are $\varepsilon,\varepsilon'$, then it is clear that $a<a'$ if and only if $\varepsilon<\varepsilon'$.
More importantly, $\delta(\varepsilon,\varepsilon')$ is simply the largest non-negative integer $\delta$ such that $\lfloor 2^{-\delta}a\rfloor \neq \lfloor 2^{-\delta}a'\rfloor$.
We generalize this observation to the definition below.
\begin{definition}
    Let $S$ be a set of at least two positive integers.
    The \emph{top splitting level} $\ell(S)$ of $S$ is the maximum non-negative integer $\ell$ such that $\{\lfloor 2^{-\ell}s\rfloor:s\in S\}$ has size at least two.
    The \emph{left subset} $S_{\Left}$ of $S$ is the set of $s\in S$ with $\lfloor2^{-\ell(S)}s\rfloor$ even, and the \emph{right subset} $S_{\Right}$ of $S$ is the set of $s\in S$ with $\lfloor2^{-\ell(S)}s\rfloor$ odd.
\end{definition}
For example, if $S=\{5,6,7,8,9\}$, then $\ell(S)=3$, $S_{\Left}=\{5,6,7\}$, and $S_{\Right}=\{8,9\}$.
Note that,
by the definition of $\ell(S)$, for any $s\in S$, there exists an integer $q_s$ and an integer $r_s$ such that $s=q_s2^{\ell(S)+1}+r_s$ where $0\le r_s< 2^{\ell(S)+1}$. Thus $\lfloor2^{-\ell(S)}s\rfloor$ is either $2q_s$ or $2q_s+1$, depending on whether $r_s\le 2^{\ell(S)}$ or not. Hence, $S_{\Left}$, $S_{\Right}$ are well-defined and always non-empty.

To get the full list of $\delta_1,\ldots, \delta_{k-1}$, we need to iteratively split the subsets into left subsets and right subsets until they all have size $1$.
If we record how the subsets split using a binary tree, we get the following definition.

\begin{definition}
    For any set of non-negative integers $S$, its \emph{binary structure} $b(S)$ is a weighted rooted ordered binary tree defined iteratively as follows.
    \begin{enumerate}
        \item If $\abs{S}=1$, then $b(S)$ is a single root vertex with weight $1$.
        \item If $\abs{S}>1$, let $b(S)$ be the binary tree so that the left subtree of the root is $b(S_{\Left})$, and the right subtree of the root is $b(S_{\Right})$. 
        The weight of the root is $\abs{S}$.
    \end{enumerate}

    Note that the weight of any vertex $v$ in $b(S)$ is equal to the number of leaves in the subtree rooted at $v$.
\end{definition}

For example, the following figure is the binary structure of $S=\{5,6,7,8,9\}$.
To compute the $\delta$'s, we can simply read the internal nodes from left to right and see which level they split.
For example, for $\varepsilon^{(1)} = (0,1,0,1)$, $\varepsilon^{(2)}=(0,1,1,0)$, $\ldots$, $\varepsilon^{(5)} = (1,0,0,1)$, we have $(\delta_1,\delta_2,\delta_3,\delta_4) = (1,0,3,0)$.

\begin{figure}[ht]
    \centering
\begin{tikzpicture}[
    scale=0.8,
    node distance = 1.5cm, 
    treenode/.style = {circle, draw, minimum size=1.2em, inner sep=1pt, fill=white},
    leafsiblingdist/.style = {xshift=0.75cm}, 
    line_label/.style = {left=0.1cm, anchor=east}
]


\node [treenode] (root) at (0,1) {5};

\node [treenode] (n3) at (-1.5,-1) {3};
\node [treenode] (n2_right) at (1.5, -2) {2};

\node [treenode] (n2_middle) at (-0.5,-2) {2};

\coordinate (LeafLevelY) at ($(n2_middle.south) - (0,1.5cm)$);

\node [treenode] (n1_farleft)   at (-2,-3) {1};
\node [treenode] (n1_midleft)   at (-1,-3) {1}; 
\node [treenode] (n1_midright)  at (0,-3) {1}; 
\node [treenode] (n1_nearright) at (1,-3) {1}; 
\node [treenode] (n1_farright)  at (2,-3) {1}; 

\draw (root) -- (n3);
\draw (root) -- (n2_right);
\draw (n3)   -- (n1_farleft); 
\draw (n3)   -- (n2_middle);
\draw (n2_middle) -- (n1_midleft);
\draw (n2_middle) -- (n1_midright);
\draw (n2_right)  -- (n1_nearright);
\draw (n2_right)  -- (n1_farright);

\node [below=0.1cm of n1_farleft]   (lab5) {5};
\node [below=0.1cm of n1_midleft]   (lab6) {6};
\node [below=0.1cm of n1_midright]  (lab7) {7};
\node [below=0.1cm of n1_nearright] (lab8) {8};
\node [below=0.1cm of n1_farright]  (lab9) {9};


\draw[dashed] (-3,1) -- (-0.3,1) node[line_label, pos=0] {$\ell=3$};
\draw[dashed] (0.3,1) -- (3,1);
\draw[dashed] (-3,0) -- (3,0) node[line_label, pos=0] {$\ell=2$};
\draw[dashed] (-3,-1) -- (-1.8,-1) node[line_label, pos=0] {$\ell=1$};
\draw[dashed] (-1.2,-1) -- (3,-1);
\draw[dashed] (-3,-2) -- (-0.8,-2) node[line_label, pos=0] {$\ell=0$};
\draw[dashed] (-0.2,-2) -- (1.2,-2);
\draw[dashed] (1.8, -2) -- (3,-2);

\end{tikzpicture}
    \caption{Binary structure of $S=\{5,6,7,8,9\}$}
    \label{fig:BinaryStructure56789}
\end{figure}
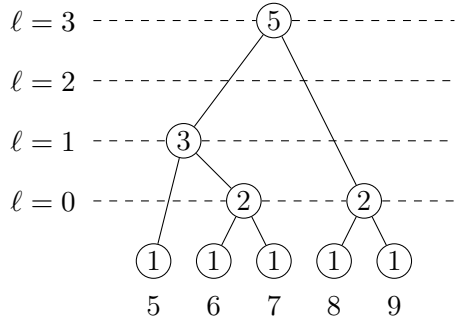

As mentioned earlier, the case where $\delta_1,\ldots, \delta_{k-1}$ is monotone is usually special.
Rephrasing this in terms of binary structures gives the following definition.

\begin{definition}
    The binary structure $b(S)$ is \emph{increasing} if all right subtrees of internal nodes are singletons, and it is \emph{decreasing} if all left subtrees of internal nodes are singletons.
    The binary structure is \emph{monotone} if it is increasing or decreasing.
    If $S$ is a set of non-negative integers whose binary structure is monotone, let $L(S)$ be iteratively defined as follows.
    \begin{enumerate}
        \item If $\abs{S}=1$, let $L(S)$ be the empty set.
        \item If $\abs{S}>1$, let $S'$ be $S_{\Left}$ if $b(S)$ is increasing, and let $S'$ be $S_{\Right}$ if $b(S)$ is decreasing.
        Then set $L(S) \eqdef L(S')\cup \{\ell(S)\}$.
    \end{enumerate}
\end{definition}
For example, $\{5,6,7,8,9\}$ does not have a monotone binary structure, whereas $\{1,2,4,8,16\}$ is decreasing with $L(\{1,2,4,8,16\}) = \{0,1,2,3\}$. Note that if $S$ is monotone, then $L(S)$ is simply the set of levels of the internal nodes in the binary structure of $S$.

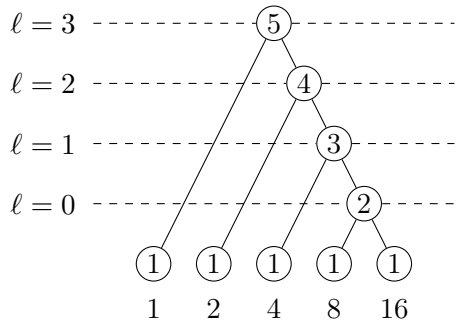
\begin{figure}[ht]
    \centering
\begin{tikzpicture}[
    scale=0.8,
    node distance = 1.5cm, 
    treenode/.style = {circle, draw, minimum size=1.2em, inner sep=1pt, fill=white},
    leafsiblingdist/.style = {xshift=0.75cm}, 
    line_label/.style = {left=0.1cm, anchor=east}
]


\node [treenode] (root) at (0,0) {5};

\node [treenode] (n4) at (0.5,-1) {4};

\node [treenode] (n3) at (1, -2) {3};

\node [treenode] (n2) at (1.5,-3) {2};

\coordinate (LeafLevelY) at ($(n2_middle.south) - (0,1.5cm)$);

\node [treenode] (n1_farleft)   at (-2,-4) {1};
\node [treenode] (n1_midleft)   at (-1,-4) {1}; 
\node [treenode] (n1_midright)  at (0,-4) {1}; 
\node [treenode] (n1_nearright) at (1,-4) {1}; 
\node [treenode] (n1_farright)  at (2,-4) {1}; 

\draw (root) -- (n4);
\draw (root) -- (n1_farleft);
\draw (n4)   -- (n3); 
\draw (n4)   -- (n1_midleft);
\draw (n3) -- (n2);
\draw (n3) -- (n1_midright);
\draw (n2)  -- (n1_nearright);
\draw (n2)  -- (n1_farright);

\node [below=0.1cm of n1_farleft]   (lab5) {1};
\node [below=0.1cm of n1_midleft]   (lab6) {2};
\node [below=0.1cm of n1_midright]  (lab7) {4};
\node [below=0.1cm of n1_nearright] (lab8) {8};
\node [below=0.1cm of n1_farright]  (lab9) {16};


\draw[dashed] (-3,0) -- (-0.3,0) node[line_label, pos=0] {$\ell=3$};
\draw[dashed] (0.3,0) -- (3,0);
\draw[dashed] (-3,-1) -- (0.2,-1) node[line_label, pos=0] {$\ell=2$};
\draw[dashed] (0.8,-1) -- (3,-1);
\draw[dashed] (-3,-2) -- (0.7,-2) node[line_label, pos=0] {$\ell=1$};
\draw[dashed] (1.3,-2) -- (3,-2);
\draw[dashed] (-3,-3) -- (1.2,-3) node[line_label, pos=0] {$\ell=0$};
\draw[dashed] (1.8,-3) -- (3,-3);

\end{tikzpicture}
    \caption{Binary structure of $S=\{1,2,4,8,16\}$}
    \label{fig:BinaryStructure124816}
\end{figure}

Finally, instead of classifying $k$-tuples using the relative order of $\delta_1,\ldots, \delta_{k-1}$, we classify them based on types of their structures.

\begin{definition}
    A \emph{type of binary structure} $T$ is a weighted rooted ordered binary tree such that each leaf vertex has a positive integer weight, and each internal vertex has weight equal to the sum of the weights of its children. 
    Its \emph{size} is the weight of its root.
    A set of positive integers $S$ is \emph{of type $T$} if $T$ can be obtained from $b(S)$ by iteratively removing leaves.
    The type $T$ is \emph{monotone} if there is some set of positive integer $S$ of type $T$ with $b(S)$ monotone.
\end{definition}

Before we introduce our reformulation of stepping-up, let us first review the classical construction given by Erd\H{o}s, Hajnal and Rado \cite{MR0202613}.
For simplicity, assume that $k=4$, and we are given a $3$-graph $G$ on $N$ vertices with small independence number and clique number.
The stepping-up of $G$ is a $4$-graph on $2^N$ vertices defined as follows: for any $\{\varepsilon^{(1)},\ldots,\varepsilon^{(4)}\}$, include it as an edge if one of the following holds:
\begin{itemize}
    \item $\delta_1,\delta_2,\delta_3$ is monotone and forms an edge in $G$.
    \item $\delta_1<\delta_2>\delta_3$.
\end{itemize}
Using our terminology, the first condition is satisfied for some set $S$ of size $4$ precisely when $b(S)$ is monotone and $L(S)$ is an edge in $G$.
With some thought, we can see that the second condition is precisely when $b(S)$ is of type $T_{2,2}$, where $T_{2,2}$ is the binary tree with weight $4$ at the root and weight $2$ at both of its children, both of which are leaves.

\begin{figure}[ht]
    \centering
\begin{tikzpicture}[
    scale=0.8,
    node distance = 1.5cm, 
    treenode/.style = {circle, draw, minimum size=1.2em, inner sep=1pt, fill=white},
    leafsiblingdist/.style = {xshift=0.75cm}, 
    line_label/.style = {left=0.1cm, anchor=east}
]


\node [treenode] (root) at (0,0) {4};

\node [treenode] (n2_left) at (-0.5,-1) {2};
\node [treenode] (n2_right) at (0.5, -1) {2};

\draw (root) -- (n2_left);
\draw (root) -- (n2_right);
\end{tikzpicture}
    \caption{Type $T_{2,2}$}
    \label{fig:BinaryStructureTypeT22}
\end{figure}
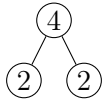
Our variant of stepping-up will also include two such conditions, which we now define in turn.
We begin with the monotone edges.
For our convenience later, we will actually ``flip $G$'' when we step up for the decreasing edges.

\begin{definition}
    Let $G$ be a $(k-1)$-graph on $\{0,1,\ldots, N-1\}$.
    Its \emph{left stepping-up} is the $k$-graph on $\{0,\ldots, 2^N-1\}$ consisting of all edges $e$ with $b(e)$ increasing and $L(e)\in E(G)$.
    Its \emph{right stepping-up} is the $k$-graph on $\{0,\ldots, 2^N-1\}$ consisting of all edges $e$ with $b(e)$ decreasing and $\{N-1-\ell: \ell\in L(e)\}\in E(G)$.
\end{definition}

Now we can define the stepping-up that we will use throughout the paper.

\begin{definition}
    For any tuple $(G_1,G_2,\cT)$ where $G_1,G_2$ are two $(k-1)$-graphs on $\{0,1,\ldots, N-1\}$ and $\cT$ is a family of types of binary structure of size $k$, its \emph{stepping up} is the $k$-graph $G$ on vertex set $\{0,\ldots, 2^N-1\}$ formed by taking the edge union
    \[G \eqdef \Tilde{G_1}\cup \Tilde{G_2}\cup G_{\cT}\]
    where $\Tilde{G_1}$ is the left stepping-up of $G_1$, $\Tilde{G_2}$ is the right stepping-up of $G_2$, and $G_{\cT}$ is the $k$-graph with 
    \[E(G_\cT) = \left\{e\in \binom{\{0,\ldots, 2^N-1\}}{k}: e\textup{ is of type }T\textup{ for some }T\in \cT\right\}.\]
\end{definition}
For example, the classical stepping-up construction discussed above for $k=4$ is precisely the stepping up of $(G,\textup{rev}(G),\{T_{2,2}\})$ where $\textup{rev}(G)$ is the hypergraph with the vertices reversed.
Note that the complement can also be described as a stepping up: it is the stepping up of $(G^c, \textup{rev}(G^c), \{T_{1,(2,1)},T_{(1,2),1}\})$ where $T_{1,(2,1)}$ and $T_{(1,2),1}$ are drawn in \cref{fig:BinaryStructureTypeT121} below.

\begin{figure}[ht]
    \centering
\begin{tikzpicture}[
    scale=0.8,
    node distance = 1.5cm, 
    treenode/.style = {circle, draw, minimum size=1.2em, inner sep=1pt, fill=white},
    leafsiblingdist/.style = {xshift=0.75cm}, 
    line_label/.style = {left=0.1cm, anchor=east}
]


\node [treenode] (root1) at (0,0) {4};

\node [treenode] (n1_left1) at (-1,-1) {1};
\node [treenode] (n31) at (1,-1) {3};
\node [treenode] (n21) at (0.5, -2) {2};
\node [treenode] (n1_right1) at (1.5, -2) {1};

\node [treenode] (root2) at (4,0) {4};

\node [treenode] (n1_left2) at (2.5,-2) {1};
\node [treenode] (n32) at (3,-1) {3};
\node [treenode] (n22) at (3.5, -2) {2};
\node [treenode] (n1_right2) at (5, -1) {1};

\draw (root1) -- (n1_left1);
\draw (root1) -- (n31);
\draw (n31) -- (n21);
\draw (n31) -- (n1_right1);
\draw (root2) -- (n1_right2);
\draw (root2) -- (n32);
\draw (n32) -- (n22);
\draw (n32) -- (n1_left2);
\end{tikzpicture}
    \caption{Type $T_{1,(2,1)}$ and Type $T_{(1,2),1}$}
    \label{fig:BinaryStructureTypeT121}
\end{figure}
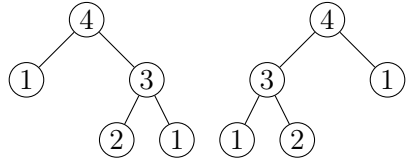

Having rephrased the classical construction, we now turn to bounding its independence number, recalling that our ultimate goal is to construct a $k$-graph with no large independent sets.
To bound the independence number of the stepped-up hypergraph, we define the following auxiliary functions.

\begin{definition}
    For any positive integers $n_1, n_2$ and any family $\cT$ of types of binary structure, let $f(n_1,n_2,\cT)$ be the maximum possible size of set $S$ of non-negative integers that does not contain any $n_1$-set with an increasing binary structure, any $n_2$-set with a decreasing binary structure, or any set of type $T$ for any $T\in \cT$.
\end{definition}

The following statement then follows easily from the definition.

\begin{lemma}\label{lemma:independence-upper-bound}
Let $G_1,G_2$ be two $(k-1)$-graphs on $\{0,1,\ldots, N-1\}$, and let $\cT$ be a family of types of binary structure of size $k$.
Let $G$ be the stepping up of $(G_1,G_2,\cT)$. 
Then
\[\alpha(G)\leq f\left(\alpha(G_1)+2,\alpha(G_2)+2,\cT\right).\]
\end{lemma}
\begin{proof}
    Let $S\subseteq \{0,\ldots, 2^N-1\}$ be an independent set in $G$.
If $S$ contains an $(\alpha(G_1)+2)$-set $S'$ with an increasing binary structure, then $\abs{L(S')}=\alpha(G_1)+1>\alpha(G_1)$, hence $L(S')$ contains an edge $e'$ in $G_1$.
This shows that $S$ contains a $k$-set $e$ with $b(e)$ increasing and $L(e)=e'$, which is a contradiction.
Therefore $S$ does not contain any $(\alpha(G_1)+2)$-set $S'$ with $b(S')$ increasing.
Similarly, $S$ does not contain any $(\alpha(G_2)+2)$-set $S'$ with $b(S')$ decreasing.
Lastly, from the way $G$ is constructed, it is clear that $S$ does not contain any sets of type $T$ for any $T\in\cT$.
As a consequence, 
$\abs{S}\leq f\left(\alpha(G_1)+2,\alpha(G_2)+2,\cT\right).$
\end{proof}
In order to apply \cref{lemma:independence-upper-bound}, we also need to compute $f(n_1,n_2,\cT)$.
This can be done by dynamic programming if $\cT$ is explicitly known.
Here we carry out the computation for two particularly useful cases.

For the first case, let $T_{a,b}$ be the type of binary structure where the children of the root are leaves with weight $a$ on the left and $b$ on the right.
We show that $f(n_1,n_2,\cT)$ is linear in $n_1+n_2$ if $\cT$ contains some type of the form $T_{a,b}$.

\begin{lemma}\label{lemma:family-of-depth-1}
    Let $n_1,n_2,k$ be positive integers with $n_1,n_2\geq 2$.
    Let $\cT$ be a family of types of binary structure of size $k$ containing types of the form $T_{a,b}$ for some $a+b=k$.
    Furthermore, let $a_{\min}$ be the minimum possible value for $a$, and $b_{\min}$ be the minimum possible value for $b$.
    Then
    \[f\left(n_1,n_2,\cT\right)\leq (a_{\min}-1)(n_2-2)+(b_{\min}-1)(n_1-2)+2k.\]
\end{lemma}
\begin{proof}
    We will proof by induction on $n_1+n_2$.
    The case $n_1+n_2=4$ holds trivially as $n_1=n_2=2$ in this case, which forces $f(n_1,n_2,\cT)=1$ as any set of size $2$ has an increasing binary structure.
    The inductive step is going to be established by the inequality
    \begin{equation}\label{eq:recursion}f(n_1,n_2,\cT)\leq \max\left\{a_{\min}-1+f(n_1,n_2-1,\cT), b_{\min}-1+f(n_1-1,n_2,\cT),2k\right\}
    \end{equation}
    Indeed, assuming the claim holds for $f(n_1,n_2-1,\cT)$ and $f(n_1-1,n_2,\cT)$, then we have
    \[a_{\min}-1+f(n_1,n_2-1,\cT) \leq (a_{\min}-1)(n_2-2)+(b_{\min}-1)(n_1-2)+2k,\]
    \[b_{\min}-1+f(n_1-1,n_2,\cT) \leq (a_{\min}-1)(n_2-2)+(b_{\min}-1)(n_1-2)+2k,\]
    and
    \[2k\leq (a_{\min}-1)(n_2-2)+(b_{\min}-1)(n_1-2)+2k.\]
    Therefore \cref{eq:recursion} and the inductive hypothesis together imply 
    \[f(n_1,n_2,\cT)\leq (a_{\min}-1)(n_2-2)+(b_{\min}-1)(n_1-2)+2k,\]
    and it remains to prove \cref{eq:recursion}.

    Suppose that $S$ is a set with no $n_1$-sets with increasing binary structures, no $n_2$-sets with decreasing binary structures, and no sets of type $T$ for any $T\in \cT$.
    Note that $S_{\Right}$ now cannot contain any $(n_2-1)$-set with decreasing binary structure: indeed, if $S'\subseteq S_{\Right}$ satisfies $\abs{S'}=n_2-1$ and $b(S')$ is decreasing, then $\{s\}\cup S'$ also has decreasing structure for any $s\in S_{\Left}$, which is a contradiction.
    Therefore 
    \[\abs{S_{\Right}}\leq f(n_1,n_2-1,\cT).\]
    Similarly
    \[\abs{S_{\Left}}\leq f(n_1-1,n_2,\cT).\]
    If $\abs{S_{\Left}}<a_{\min}$ or $\abs{S_{\Right}}<b_{\min}$, then \cref{eq:recursion} now follows immediately.
    Therefore let us now assume $\abs{S_{\Left}}\geq a_{\min}$ and $\abs{S_{\Right}}\geq b_{\min}$.
    Since $S$ does not contain any set of type $T_{a_{\min},k-a_{\min}}$ or $T_{k-b_{\min},b_{\min}}$, we see that $\abs{S_{\Right}}<k-a_{\min}$ and $\abs{S_{\Left}}<k-b_{\min}$.
    This shows that $\abs{S}<2k-a_{\min}-b_{\min}<2k$, as desired.
\end{proof}

We will use \cref{lemma:family-of-depth-1} to get a better quantitative dependence when stepping up.
However, if one is only interested in understanding the tower heights of Ramsey numbers, then it suffices to use the much cheaper bound below.

The \emph{depth} of a binary tree is the number of edges in the longest path from the root to a leaf.

\begin{lemma}\label{lemma:f-upper-bound}
Let $k,d$ be two non-negative integers.
If $T$ is a type of binary structure of size $k$ and depth $d$, then
\[f(n_1,n_2,\{T\})=O_k( (n_1+n_2)^{d}).\]
\end{lemma}

\begin{proof}
We induct on $d$.
If $d=0$, then any set of non-negative integers of size $\abs{T}$ is of type $T$, which shows that $f(n_1,n_2,T)\leq \abs{T}-1=O_{\abs{T}}((n_1+n_2)^{d})$.
Now suppose that the claim holds for all smaller depths.
Let $S$ be a set of non-negative integers that avoids $n_1$-sets with increasing binary structures, $n_2$-sets with decreasing binary structures, and $\abs{T}$-sets of type $T$.
Let $T_{\Left}$ and $T_{\Right}$ be the left and right subtrees of the root of $T$, respectively.
Suppose that $\abs{S}\geq 2$.
As in the previous proof, we know that $S_{\Left}$ does not contain any $(n_1-1)$-sets with increasing binary structures.
Similarly, $S_{\Right}$ does not contain any $(n_2-1)$-sets with decreasing binary structures.

Moreover, it cannot simultaneously hold that $S_{\Left}$ contains a set of type $T_{\Left}$, and that $S_{\Right}$ contains a set of type $T_{\Right}$, as these would combine to give a set of type $T$ in $S$. If 
$S_{\Left}$ does not contain a set of type $T_{\Left}$, then
\[\abs{S} = \abs{S_{\Left}}+\abs{S_{\Right}}\leq f(n_1-1,n_2,\{T_{\Left}\})+f(n_1,n_2-1,\{T\}).\]
Similarly, if $S_{\Right}$ does not contain a set of type $T_{\Right}$, then

\[\abs{S} = \abs{S_{\Left}}+\abs{S_{\Right}}\leq f(n_1-1,n_2,\{T\})+f(n_1,n_2-1,\{T_{\Right}\}).\]
Combining the two cases, we have
\begin{align*}
    f(n_1,n_2,\{T\})\leq \max&\left\{f(n_1-1,n_2,\{T_{\Left}\})+f(n_1,n_2-1,\{T\}),\right.\\&\left.f(n_1-1,n_2,\{T\})+f(n_1,n_2-1,\{T_{\Right}\})\right\}.
\end{align*}
For any $n\geq 2$, set $g(n,T)\eqdef \max_{n_1+n_2=n}f(n_1,n_2,\{T\})$.
Then the above inequality gives
\[g(n,T)\leq g(n-1,T)+\max\left\{g(n-1,T_{\Left}), g(n-1,T_{\Right})\right\},\]
and hence
\[g(n,T)\leq g(2,T)+\sum_{i=2}^{n-1}\max\left\{g(i,T_{\Left}), g(i,T_{\Right})\right\}.\]

Note that $g(2,T)=f(1,1,T)=0$ and 
$$\max\{g(i,T_{\Left}),g(i,T_{\Right})\}\le\max\{ g(n,T_{\Left}),g(n,T_{\Right})\}=O_{\abs T}(n^{d-1})$$ by the inductive hypothesis.
Therefore $g(n,T)=O_{\abs T} (n^{d})$ and so $f(n_1,n_2,\{T\})=O_{\abs T} ((n_1+n_2)^{d})$, which closes the induction.
\end{proof}

We briefly remark that from \cref{lemma:family-of-depth-1}, it is easy to see that the stepping-up construction for $k=4$ of Erd\H{o}s, Hajnal and Rado only grows the independence number at most linearly.
To bound the clique number, we can apply \cref{lemma:f-upper-bound} to see that the independence number of the complement only grows at most quadratically as $T_{1,(2,1)}$ has depth $2$.
To show that the clique number grows at most linearly, one would need to compute $f(n_1,n_2,\{T_{1,(2,1)},T_{(1,2),1}\})$, and it is not hard to see that once we take $T_{(1,2),1}$ into consideration, this function does indeed grow linearly in $n_1$ and $n_2$.
We remark too that one can similarly rephrase stepping-up for general uniformities using our binary structures framework,
but we omit the details as this does not improve on the classical construction.

\section{Linear hypergraphs}\label{sec:linear}
We now turn to the proof of our technical result, \cref{prop:step up linear}. The main remaining difficulty is the construction of the linear $k$-graph; the following result gives the main property we need of this construction, namely that it is avoided by the stepping-up constructions discussed in the previous section. 

\begin{theorem}\label{thm:linear-stepping-up}
    Let $k\geq 4$, and let $H$ be a linear $(k-1)$-graph with no isolated vertices.
    Then there exists a linear $k$-graph $H'$ such that the following holds.
    Suppose that $G$ is a $(k-1)$-graph that is $H$-free, and $\cT$ is a family of types of binary structure of size $k$ that are not monotone.
    Then the stepping up $\Tilde{G}$ of $(G,G,\cT)$ is $H'$-free. 
\end{theorem}

We remark that we need $k\geq 4$ because any type of binary structure of size at most $3$ is monotone. Given \cref{thm:linear-stepping-up}, we can readily complete the proof of \cref{prop:step up linear} and hence of \cref{thm:main}.
\begin{proof}[Proof of \cref{prop:step up linear}]
    For any linear $(k-1)$-graph $H$, let $H'$ be the linear $k$-graph obtained from \cref{thm:linear-stepping-up}.
    By definition, there exists a $(k-1)$-graph $G$ on $\left\{0,1,\ldots, r(H,K_n^{(k-1)})-2\right\}$ that is $H$-free with $\alpha(G)\leq n-1$.
    Let $\cT$ be $\{T_{2,k-2},T_{k-2,2}\}$.
    As $k\geq 4$, we know that both children have weight at least $k-2\geq 2$ and so $T_{2,k-2}$ and $T_{k-2,2}$ are not monotone.
    
    Let $\Tilde{G}$ be the stepping up of $(G,G,\cT)$.
    Then we know that $\Tilde{G}$ is $H'$-free and 
    \[\alpha(\Tilde{G})\leq 2\alpha(G)+2k<2n+2k\]
    by \cref{lemma:independence-upper-bound,lemma:family-of-depth-1}.
    This shows that
    \[r(H', K_{2n+2k}^{(k)})> \lvert{V(\Tilde{G})}\rvert = 2^{r(H,K_n^{(k-1)})-1},\]
    as desired.
\end{proof}
We now turn to the proof of \cref{thm:linear-stepping-up}, which occupies the remainder of this section.
We first make the following two definitions, which arise naturally when analyzing stepping-up constructions.
\begin{definition}
    A \textit{dyadic partition} of a set $A$ is a partition $A=A_1\sqcup A_2 \sqcup \cdots \sqcup A_t$ with the property that $|A_i| \le \frac{1}{2} (|A_i| + \cdots + |A_t|)$ for all $i < t$ and $|A_t| = 1$.
\end{definition}

The following lemma captures some properties we need of our construction of a linear $k$-graph.

\begin{lemma}\label{lem:one shot construction}
    Let $k \geq 3$, and fix an ordered linear $k$-graph $F^<$. There exists a linear $k$-graph $H'$ with the following property. For every dyadic partition $V(H')=A_1 \sqcup \dots \sqcup A_t$, and for every two-coloring $\beta:[t] \to \{L,R\}$, and for every ordering $\pi$ of $V(H')$, there exists an ordered monochromatic transversal copy of $F^<$. That is, there is a copy of $F^<$ each of whose vertices lie in a different part $A_i$, such that all these $A_i$ receive the same color under $\beta$, and such that the order of these vertices under $\pi$ agrees with the fixed ordering of $F^<$.
\end{lemma}

We remark that \cref{lem:one shot construction} works for $k=3$.
However, our stepping-up construction only works for $k\geq 4$, as in the statement of \cref{thm:linear-stepping-up}. We will apply \cref{lem:one shot construction} with $F^<$ chosen to be a certain $k$-graph derived from a given $(k-1)$-graph, which we now define.

\begin{definition}
    Let $H$ be a $(k-1)$-graph. Its \emph{expansion} $H^+$ is the $k$-graph obtained from $H$ by adding to each $e \in E(H)$ a new vertex $v_e$, which participates in no other edges.

    An \emph{ordered expansion} of $H$ is any ordering of $V(H^+)$ in which, for all $e \in E(H)$, the vertex $v_e$ is the smallest vertex of $e\cup \{v_e\}$.
\end{definition}
For example, we may obtain an ordered expansion by putting all the new vertices $v_e$ first, followed by all the old vertices of $H$. 
Using \cref{lem:one shot construction} and the properties of the stepping-up construction discussed above, we are now ready to prove \cref{thm:linear-stepping-up}. 

\begin{proof}[Proof of \cref{thm:linear-stepping-up}]
    Let $F^<$ be an ordered expansion of $H$, and let $H'$ be the $k$-graph given by \cref{lem:one shot construction}. We will show that $H'$ satisfies the requirement of \cref{thm:linear-stepping-up}.
    Suppose for the sake of contradiction that $\Tilde{G}$ contains a copy of $H'$.
    In this case, we will think of $V(H')$ as a subset of $V(\Tilde{G})$ and produce a dyadic partition $A_1\sqcup\cdots\sqcup A_t$ with a coloring $\beta:[t]\to\{L,R\}$ of $V(H')$ as follows.
    Starting with $B_0 = V(H')$, suppose that we have defined $B_i$ and $A_1,\ldots, A_i$ so that $V(H') = A_1\sqcup\cdots\sqcup A_i\sqcup B_i$.
    If $B_i$ is a singleton, then we set $t=i+1, A_{i+1}=B_i$ and then terminate.
    Therefore we will now assume that $\abs{B_i}\geq 2$, and so its left and right subsets are well-defined and non-empty.
    Let $(A_{i+1},B_{i+1})$ be a permutation of $((B_i)_{\Left}, (B_i)_{\Right})$ so that $\abs{A_{i+1}}\leq \abs{B_{i+1}}$.
    Moreover, we set $\beta(i+1)=L$ if $A_{i+1}$ is the left subset of $B_i$, and $\beta(i+1)=R$ otherwise.
    Finally, we increase $i$ by $1$ and continue the process until it terminates.
    To check that $A_1\sqcup\cdots\sqcup A_t$ is indeed a dyadic partition of $V(H')$, note that the procedure guarantees that it is a partition with $\abs{A_t}=1$.
    In addition, for all $1\leq i<t$, we have that $\abs{A_i}\leq \frac{1}{2}\left(\abs{A_i}+\abs{B_i}\right) = \frac{1}{2}\left(\abs{A_i}+\cdots+\abs{A_t}\right)$.

    Having constructed the dyadic partition and the coloring, we now construct the ordering $\pi$ of $V(H')$.
    For any $v\in V(H')$, let $i(v)$ be the unique $i$ such that $v\in A_i$.
    Then for any $u,v\in V(H')$, we set $u>_{\pi} v$ if $i(u)<i(v)$.
    If $i(u)=i(v)$, then we arbitrarily order $u$ and $v$ with respect to $\pi$.

    By \cref{lem:one shot construction}, there is an ordered monochromatic transversal copy of $F^<=H^+$.
    We will use this copy to find a copy of $H$ in $G$, which gives a contradiction.
    First assume that the ordered monochromatic transversal copy of $H^+$ is monochromatic of color $R$.
    We will again think of $V(H^+)$ as a subset of $V(H')$, which is a subset of $V(\Tilde{G})$.
    By the definition of $H^+$, there is an injection $\varphi: V(H)\to V(H^+)$ such that for each edge $e\in E(H)$, we have that $\varphi(e)$ is contained in an edge $e'$ in $V(H^+)$ and $\varphi(e)$ does not contain the smallest element $v_e$ (with respect to $\pi$) in $e'$.
    Now we define $\varphi^*:V(H)\to V(G)=\{0,1,\ldots, N-1\}$ as $\varphi^*(v) = \ell\left(B_{i(\varphi(v))-1}\right)$, the top splitting level of $B_{i(\varphi(v))-1}$.
    We will show that this embeds $H$ into $G$. 

    For each $e\in H$, we can order the vertices in $e$ as $v_1,\ldots, v_{k-1}$ so that $i(\varphi(v_1))>\cdots > i(\varphi(v_{k-1}))$. (Recall that they are distinct as the copy of $H'$ is transversal).
    By the definition of $\pi$, we have that $v_e<_{\pi}\varphi(v_1)<_{\pi}\cdots <_{\pi} \varphi(v_{k-1})$, which shows that $i(v_e)>i(\varphi(v_1))$ again by the definition of $\pi$.
    Now for every $j\in[k-1]$, we will show that the top splitting level of $\{v_e,\varphi(v_1),\ldots, \varphi(v_j)\}$ is $\varphi^*(v_j)$ and its right subset is the singleton $\{\varphi(v_j)\}$.
    Note that by construction, it is clear that $B_{i(\varphi(v_j))-1}\supseteq \{v_e,\varphi(v_1),\ldots, \varphi(v_j)\}$.
    Moreover, as $\beta(i(\varphi(v_j)))=R$, the left subset of $B_{i(\varphi(v_j))-1}$ is $B_{i(\varphi(v_j))}$, which contains $v_e,\varphi(v_1)\ldots, \varphi(v_{j-1})$ as $i(v_e)>i(\varphi(v_1))>\cdots >i(\varphi(v_{j-1}))> i(\varphi(v_j))$; whereas its right subset is $A_{i(\varphi(v_j))}$, which contains $\varphi(v_j)$ by definition.
    Therefore the top splitting level of $\{v_e,\varphi(v_1),\ldots, \varphi(v_j)\}$ is $\ell(B_{i(\varphi(v_j))-1}) = \varphi^*(v_j)$, and its right subset is $\{\varphi(v_j)\}$.

    This immediately implies that $\{v_e\}\cup \varphi(e)$ has an increasing binary structure with $L(\{v_e\}\cup \varphi(e))$ is $\varphi^*(e)$.
    Note that for any $T\in \cT$, as $T$ is not monotone, we know that $\{v_e\}\cup \varphi(e)$ is not of type $T$.
    We also know that $\{v_e\}\cup\varphi(e)$ is not decreasing as $k\geq 3$.
    Therefore $\{v_e\}\cup\varphi(e)$ must be in the left stepping-up of $G$, and so $\varphi^*(e) = L(\{v_e\}\cup \varphi(e))\in E(G)$ by definition.
    As a consequence, $\varphi^*$ is an embedding from $H$ to $G$, which is a contradiction with the assumption that $G$ is $H$-free.
    Therefore we have reached a contradiction from assuming that there is a monochromatic ordered transversal copy of $H^+$ in $H$ with color $R$.

    The other case where the ordered transversal copy of $H^+$ is monochromatic of color $L$ can be dealt with analogously.
    We swap the roles of left and right, and redefine $\varphi^*(v)$ as $N-1-\ell(B_{i(\varphi(v))-1})$, and
    the rest of the argument holds verbatim.
    As we reach a contradiction either way, we see that $\Tilde{G}$ must be $H'$-free.
\end{proof}
\subsection{Properties of dyadic partitions}
All that remains is to prove \cref{lem:one shot construction}, which we now turn to.
In this subsection, we record two useful facts we will need about the sizes of various parts in a dyadic partition, and which we use in the next subsection to complete the proof..
The first demonstrates that it is not possible to put too many elements of $A$ into a small collection of the sets $A_i$.
\begin{lemma}\label{lemma:monochromatic-greedy}
Let $A=A_1\sqcup A_2 \sqcup \cdots \sqcup A_t$ be a dyadic partition. Then for every $I\subseteq [t]$, 
$$
|A|-\sum_{i\in I}|A_i|\ge \left\lfloor\frac{|A|}{2^{|I|}}\right\rfloor.
$$
\end{lemma}

\begin{proof}
Let $I=\{i_1,i_2,\dots,i_s\}$ where $i_1<i_2<\dots<i_s$, and let $b_j=|A|-\sum_{k=1}^j|A_{i_k}|$. Suppose first that $i_s<t$. By the definition of a dyadic partition, we have $b_1\ge |A|/2$. For $j\ge 2$, suppose $b_{j-1}\ge |A|/2^{j-1}$. Then 
\begin{align*}
b_j&=|A|-\sum_{k=1}^{j}|A_{i_k}|=\frac{1}{2}\left(|A|-\sum_{k=1}^{j-1}|A_{i_k}|\right)+\frac{1}{2}\left(|A|-\sum_{k=1}^{j-1}|A_{i_k}|\right)-|A_{i_j}|\\
&\ge \frac{b_{j-1}}{2}+\frac{1}{2}\left(|A|-\sum_{k=1}^{j-1}|A_{i_k}|-\sum_{i=i_j}^t|A_i|\right)\ge \frac{|A|}{2^j},
\end{align*}
where
the first inequality uses the fact that $|A_{i_j}|\le \frac{1}{2}\sum_{i=i_j}^t|A_i|$, which follows from the definition of a dyadic partition and the fact that $i_j<t$, and the second inequality uses the inductive hypothesis and the fact that the term in parentheses is non-negative.

This completes the proof in case $i_s<t$. Additionally, the result is vacuous if $|A|<2^{|I|}$. So we may assume that $i_s=t$ and $|A|/2^{|I|-1}>2$. We then apply the argument above to $I \setminus \{t\}$ to conclude that
\[
|A|-\sum_{i\in I}|A_i| = |A|-\sum_{i \in I \setminus\{t\}} |A_i| - |A_t| \geq \frac{|A|}{2^{|I|-1}}-1 \geq \left\lfloor \frac{|A|}{2^{|I|}}\right\rfloor,
\]
where we use the fact that $|A_t|=1$ in a dyadic partition, as well as the inequality $x-1 \geq \lfloor x/2\rfloor$, valid for all real $x \geq 1$.
\end{proof}

Our next lemma, which will be crucial in the analysis of the stepping-up construction, is a simple corollary of \cref{lemma:monochromatic-greedy}.
\begin{lemma}\label{lemma:dyadic lower bound}
Let $A=A_1\sqcup A_2 \sqcup \cdots \sqcup A_t$ be a dyadic partition. For every $a\in A$, let $i(a)$ be the unique index $i$ such that $a\in A_i$. For every integer $s<t/2$ and every 2-coloring $\beta:[t]\rightarrow\{L,R\}$, there exists $X\in \{L, R\}$ such that for every $I\subseteq[t]$ with $|I|=s-1$, 
$$
|\{a\in A~\big|~i(a)\not\in I,\beta(i(a))=X\}|\ge \left\lfloor\frac{|A|}{2^{2s-1}}\right\rfloor.
$$
\end{lemma}

\begin{proof}
Suppose for contradiction that for each $X\in\{L,R\}$ there exists $I_X$ with $|I_X|=s-1$ such that
$$
|\{a\in A~\big|~i(a)\not\in I_X,\beta(i(a))=X\}|< \left\lfloor\frac{|A|}{2^{2s-1}}\right\rfloor.
$$
Then we have
$$
|A|-\sum_{i\in I_L\cup I_R}|A_i|\le \sum_{X\in\{L,R\}}|\{a\in A~\big|~i(a)\not\in I_X,\beta(i(a))=X\}|<2 \left\lfloor\frac{|A|}{2^{2s-1}}\right\rfloor \leq \left\lfloor \frac{|A|}{2^{2s-2}}\right\rfloor.
$$
Note that $|I_L\cup I_R|\le 2s-2$, hence by \Cref{lemma:monochromatic-greedy},
$$
|A|-\sum_{i\in I_L\cup I_R}|A_i|\ge \frac{|A|}{2^{|I_L\cup I_R|}}\ge \left\lfloor\frac{|A|}{2^{2s-2}}\right\rfloor.
$$
This gives a contradiction, and therefore completes the proof.
\end{proof}

\subsection{Proof of Lemma \ref{lem:one shot construction}}
In this section we prove \cref{lem:one shot construction}, thus completing the proof of \cref{thm:main}.

Let us define an \emph{oriented $s$-graph} to be an $s$-graph $H_0$, as well as a bijection $\psi_e:e \to [s]$ for every edge $e \in E(H_0)$. Note that in the case $s=2$, this precisely recovers the definition of an oriented graph. We say an oriented $s$-graph is \emph{linear} if its underlying unoriented $s$-graph is linear.

In the proof of \cref{lem:one shot construction}, we will use the following statement, which is simply the same statement in case $F^<$ is an ordered $s$-uniform edge.
\begin{lemma}\label{lem:oriented s-uniform}
    Let $s \geq 3$. There exists a linear oriented $s$-graph $H_0$ with the following property. For every dyadic partition $V(H_0)=A_1 \sqcup \dots \sqcup A_t$, and for every two-coloring $\beta:[t] \to \{L,R\}$, and for every ordering $\pi$ of $V(H_0)$, there exists an ordered monochromatic transversal edge. That is, there exists an edge $e \in E(H_0)$ which goes between $s$ distinct parts $A_i$, all of which receive the same color under $\beta$, such that the ordering given by $\psi_e$ agrees with the ordering induced by $\pi$ on $e$.
\end{lemma}

\begin{proof}
    Let $n$ be a sufficiently large integer with respect to $s$, and let $c>0$ be a sufficiently small constant. Let $p = cn^{2-s}$. Let $H_0 \sim G_s(n,p)$ be a random $s$-graph on $k$ vertices in which each $s$-tuple is made an edge with probability $p$, independently over all choices. Additionally, we make $H_0$ into an oriented $s$-graph in a uniformly random way: each edge $e \in E(H_0)$ selects a uniformly random bijection $\psi_e:e \to [s]$, independently of all other choices. Let $\cE$ be the event that there is an ordered monochromatic transversal edge in $H_0$ for every dyadic partition of $V(H_0)$, every two-coloring of $[t]$, and for every ordering of $V(H_0)$. We begin by estimating the probability of $\cE$, and then show that with positive probability we can ensure both that $\cE$ holds and that $H_0$ is linear.

    Let us fix a dyadic partition $V(H_0)=A_1 \sqcup \dots \sqcup A_t$, a coloring $\beta:[t] \to \{L,R\}$, and an ordering $\pi$ of $V(H_0)$. Note that there are at most $n^n$ dyadic partitions of a set of order $n$ (because every such partition yields a unique function $[n] \to [t]$, and $t \leq n$), there are $2^t \leq 2^n$ choices for $\beta$, and there are $n! < n^n$ orderings of $V(H_0)$. So in total we need to union-bound over at most $(2n)^{2n}$ choices for these three data.

    For any $a\in A_1\sqcup\cdots \sqcup A_t$, let $i(a)$ be the unique $i$ such that $a\in A_i$.
    We begin by estimating the number of potential monochromatic transversal edges. This number is exactly
    \[T \eqdef \frac{1}{s!}\#\{a_1,\ldots, a_s\in V(H_0): i(a_1),\ldots, i(a_s)\textup{ distinct, and } \beta(i(a_1))=\cdots=\beta(i(a_s))\}.\]
    We first claim that $T = \Omega_s(n^s)$, which we prove as follows.
    Let $X\in \{L,R\}$ be such that the conclusion of \cref{lemma:dyadic lower bound} holds\footnote{Recall that $n$ is sufficiently large, hence $t \geq \log n$ is also sufficiently large, so we may assume $t>2s$ in order to apply \cref{lemma:dyadic lower bound}}.
    Then for any $i_1,\ldots, i_{s-1}\in [t]$, there exist at least $\Omega_s(n)$ choices for $a_s$ with $i(a_s)\notin \{i_1,\ldots, i_{s-1}\}$ and $\beta(i(a_s))=X$.
    This  shows that for any $a_1,\ldots, a_m$ with $i(a_1),\ldots, i(a_m)$ distinct, $\beta(i(a_1))=\cdots=\beta(i(a_m))=X$ and $m<s$, there are at least $\Omega_s(n)$ ways to extend this sequence by $a_{m+1}$ while maintaining the properties.
    This shows $T=\Omega_s(n^s)$, as claimed.
    
    Now, let $Z$ be the random variable counting the number of monochromatic transversal edges in $H_0$ which are correctly ordered, i.e.\ whose ordering $\psi_e$ agrees with $\pi$. 
    Recall that the random ordering is chosen independently of the other random choices, showing that $Z$ is distributed as a binomial random variable with parameters $T$ and $p/s!$. Therefore,
    \[
    \PP(Z=0) = (1-p/s!)^T \leq e^{-pT/s!} \leq e^{-\Omega_s(pn^s)} = e^{-\Omega_s(cn^2)}.
    \]
    Applying the union bound over the at most $(2n)^{2n}$ choices for the dyadic partition as well as $\beta$ and $\pi$, we see that with probability at least $1-e^{-\Omega_s(cn^2)}$, we have the claimed property for all possible choices of these three data. In other words, $\PP(\cE) \geq 1-e^{-\Omega_s(cn^2)}$.

    It remains to ensure that $H_0$ is linear. For a pair of vertices $u,v$, let $\cE_{u,v}$ be the event that $u$ and $v$ are contained in at most one edge of $H_0$. By the union bound, $\PP(\overline{\cE_{u,v}}) \leq 4^s n^{2s-4} p^2=c^2 4^s$, as we have at most $n^{2s-4}$ ways of picking the remaining vertices, at most $4^s$ ways of partitioning them into two edges, and probability $p^2$ that both edges appear. Now, we observe that each $\cE_{u,v}$ is a down-event, so by Harris's inequality (or the FKG inequality), we have
    \[
    \PP \left( \bigwedge_{\{u,v\} \in \binom{V(G)}2} \cE_{u,v}\right) \geq \prod_{u,v} (1-c^2 4^s) = (1-c^2 4^s)^{\binom n2} \geq e^{-O_s(c^2 n^2)}.
    \]
    Recall that $\PP(\cE)\geq 1-e^{-\Omega_s(cn^2)}$. Therefore, for $c$ sufficiently small in terms of $s$, and for $n$ sufficiently large, we have that the sum of these two probabilities is larger than $1$. Hence, with positive probability, $H_0$ is linear and contains an ordered monochromatic transversal edge for every dyadic partition, every coloring, and every ordering.
\end{proof}
\begin{remark}
    The same proof shows that we can actually take $H_0$ to have (Berge) girth $>g$, for any fixed $g$. The linear case is simply this statement for $g=2$. Indeed, another application of Harris's inequality shows that the probability that $H_0$ does not contain a Berge cycle of length at most $g$ is also at least $e^{-O_{s,g}(c^2 n^2)}$.
\end{remark}
We are finally ready to prove \cref{lem:one shot construction}.
\begin{proof}[Proof of \cref{lem:one shot construction}]
    Let $s=\abs{F^<}$. Let $H_0$ be the oriented $s$-graph given by \cref{lem:oriented s-uniform}. Obtain a $k$-graph $H'$ by placing a copy of $F^<$ in each $s$-uniform edge of $H_0$, in an order-respecting way. That is, given an edge $e \in E(H_0)$, we insert $F^<$ into $e$ by mapping the $i$th vertex in the ordering of $F^<$ to $\psi_e^{-1}(i)$, for all $1 \leq i \leq s$. Note that since $F^<$ and $H_0$ are both linear, so is $H'$. Additionally, by \cref{lem:oriented s-uniform}, we know that in any dyadic partition, coloring, and ordering of $H_0$, there is a monochromatic correctly-ordered transversal edge. This precisely yields the claimed monochromatic correctly-ordered transversal copy of $F^<$ in $G$.
\end{proof}
\begin{remark}
    As above, if we assume that $F^<$ has girth $>g$, then we can ensure that $G$ also has girth more than $g$, since any short Berge cycle in $G$ must either stay entirely within one edge of $H_0$, or traverse at least $g$ edges in $H_0$, and both of these are ruled out by the assumption that $F$ and $H_0$ have girth $>g$.
\end{remark}

\section{Concluding remarks}
One of the most fascinating special cases of a linear hypergraph is the Fano plane $F_3$; this is the seven-vertex $3$-graph whose hyperedges are the lines in the projective geometry $\PG(2,2)$ over $\FF_2$. The growth rate of $r(F_3,K_n\up 3)$ remains wide open: the best known lower bound is polynomial and the best known upper bound is exponential, and it would be very interesting to narrow the gap. In particular, $F_3$ is not iterated tripartite, so the conjecture of \cite{2411.13812} predicts that $r(F_3,K_n\up 3)$ grows super-polynomially; proving or disproving this is a natural step towards resolving the conjecture.

There is a natural $4$-uniform analogue of the Fano plane, namely the 13-vertex linear $4$-graph $F_4$ whose hyperedges are the lines in the projective geometry $\PG(2,3)$ over $\FF_3$. $F_4$ is again not iterated $4$-partite; curiously, however, it is quite straightforward to show that $r(F_4,K_n\up 4)$ grows super-polynomially.
\begin{proposition}\label{prop:F4}
    We have $r(F_4,K_n\up 4) > 2^{n-2}$. 
\end{proposition}
\begin{proof}
    Let $G_k$ be the $4$-graph iteratively defined\footnote{We remark that $G_k$ can equivalently be defined as the stepping-up $(G_1,G_2,T_{2,2})$ where $G_1$ and $G_2$ are empty 3-graphs.} as follows: $G_1$ is a single edge, and $G_k$ is obtained from the disjoint union of two copies of $G_{k-1}$ by adding in all edges intersecting each copy in exactly $2$ vertices. Note that $\alpha(G_k) = \alpha(G_{k-1})+1$, since every independent set in $G_k$ cannot contain two vertices from both copies of $G_{k-1}$. Together with the base case $\alpha(G_1)=3$, we see that $\alpha(G_k) = k+2$ for all $k$. Since $G_k$ has $2^{k+1}$ vertices, the claimed result follows if we can show that $G_k$ is $F_4$-free for all $k$.

    If $F_4 \subseteq G_k$ for some $k$, then in particular there is a partition of $V(F_4)$ into two non-empty parts $A\cup B$ such that every edge is either fully contained in a part, or else intersects each part in exactly two vertices, and it can be checked by computer or minor casework that there is no such partition. For example, the following \textsc{SageMath} code does the job.
    \begin{verbatim}
    import sage.combinat.designs
    F = designs.ProjectiveGeometryDesign(2,1,GF(3))
    for A in Subsets(F.ground_set()):
        if all(len(A.intersection(Set(e))) in [0,2,4] for e in F.blocks()):
            print(A)
    \end{verbatim}
    This code outputs only $A = \varnothing$ and $A=V(F_4)$, hence there is no such partition into non-empty parts.
\end{proof}
Note that \cref{thm:main} yields a linear $4$-graph $H$ with $r(H,K_n\up 4) \geq 2^{2^{(\log n)^C}}$, which is larger than the lower bound guaranteed by \cref{prop:F4}. However, the construction of $H$ in \cref{thm:main} is probabilistic, whereas \cref{prop:F4} gives a fully explicit linear $4$-graph, namely $F_4$. We remark that by using this as our base case, one can similarly construct explicit linear $k$-graphs $H$ with $r(H,K_n\up k) \geq \twr_{k-2}(\Omega_k(n))$. Indeed, while the proof of \cref{lem:one shot construction} is probabilistic, it can be made explicit by using the ordered version \cite[Theorem 3.1]{2502.09830} of the recent girth Ramsey theorem of Reiher and R\"odl \cite{2308.15589}. One can use this extremely powerful machinery to provide an alternative proof of \cref{lem:one shot construction}; as the proof of the girth Ramsey theorem is fully deterministic, this would yield an explicit (but extraordinarily large and complicated) linear $k$-graph $H$ with $r(H,K_n\up k) \geq \twr_{k-2}(\Omega(n))$.

Finally, it would be very interesting to improve the lower bound in \cref{thm:linear LB 3-uniform} to $r(H,K_n\up 3) \geq 2^{n^c}$ for some absolute constant $c>0$, and some linear $3$-graph $H$. By repeated applications of \cref{prop:step up linear}, this would immediately imply the existence of a linear $k$-graph $H$ with $r(H,K_n\up k) \geq \twr_{k-1}(n^c)$, matching the upper bound in \eqref{eq:twr UB} up to the constant. 

\bibliographystyle{yuval}
\bibliography{linear.bib}

\end{document}